\documentclass[12pt]{amsart}
\usepackage{amssymb}
\usepackage{amscd}
\usepackage{eucal}
\usepackage{mathrsfs}
\usepackage{a4wide}
\usepackage[all]{xy} % xy-pic

\def\con{\subseteq}
\def\from{\colon}
\def\cont{\mathfrak{c}}

\newcommand\CH{\ensuremath{\mathsf{CH}}}

\def\mathcal#1{\mathscr{#1}}

\def\interval{\mathbb{I}}

\newtheorem{theorem}{Theorem}[section]

\newtheorem{lemma}[theorem]{Lemma}

\theoremstyle{definition}

\newtheorem{example}[theorem]{Example}
\theoremstyle{remark}

\numberwithin{equation}{section}

\makeatletter
\def\range{\mathop{\operator@font range}\nolimits}
\makeatother

\begin{document}

\author{Istv\'an Juh\'asz}
\address{Alfr\'ed R\'enyi Institute of Mathematics, Hungarian Academy of Sciences}
\email{juhasz@renyi.hu}
\author{Jan van Mill}
\address{KdV-institute for Mathematics, University of Amsterdam}
\email{j.vanMill@uva.nl}

\title[Preserving maps]{On maps preserving connectedness and/or compactness}

\date{\today}


\begin{abstract}
We call a function $f\from X\to Y$ \emph{$P$-preserving} if, for every subspace $A \subset X$
with property $P$, its image $f(A)$ also has property $P$.
Of course, all continuous maps are
both compactness- and connectedness-preserving and the natural question
about when the converse of this holds, i.e. under what conditions is such a map continuous, has a long history. 

Our main result is that any non-trivial product function, i.e. one
having at least two non-constant factors, that has connected domain, $T_1$ range, and is connectedness-preserving
must actually be continuous.
The analogous statement badly fails if we replace in it the occurrences of ``connected" by ``compact".
We also present, however, several interesting results and examples concerning maps that are compactness-preserving
and/or continuum-preserving.
\end{abstract}

\thanks{The research and preparation on this paper was supported by NKFIH grant no. K113047. It derives from the authors' collaboration at the R\'enyi Institute in Budapest in
May 2017. }

\vskip -4cm

\maketitle

\hfill \emph{This paper is dedicated to the memory of Bohuslav Balcar.}

\section{Introduction}\label{introduction}
Let $P$ be any topological property and  $X,\,Y$ be topological spaces. We call a function $f\from X\to Y$ \emph{$P$-preserving} if the image $f(A)$ of every subspace $A \subset X$
with property $P$ has property $P$ as well.
%Moreover, we call a function $f\from X\to Y$ \emph{connectedness-preserving} if the image under $f$ of every connected subspace of $X$ is connected in $Y$.
A \emph{preserving} function is one that is both compactness- and connectedness-preserving. Of course, all continuous maps are preserving and the natural question
about when the converse of this holds, i.e. under what conditions is a preserving map continuous, has been treated by  numerous authors.
More details about the history of this question can be found in Gerlits, Juh\'asz, Soukup and Szentmikl\'ossy~\cite{GerlitsJuhaszSoukupSzentmiklossy04}.

White~\cite{White71} proved that if the Tychonov space $X$ is not locally connected at a point $p\in X$, then there exists a preserving function from $X$
into the interval $\interval = [0,1]$ which is not continuous at $p$.

McMillan~\cite{McMillan70} proved that if $X$ is a locally connected Fr\'echet $T_2$-space, then every preserving map of $X$ into a $T_2$-space is continuous. (The Hausdorffness of $X$ is actually
superfluous, as was shown in ~\cite{GerlitsJuhaszSoukupSzentmiklossy04}.)

The main result  of \cite{GerlitsJuhaszSoukupSzentmiklossy04} is that if $X$ is a product of connected linearly ordered spaces and~$Y$ is regular, then each preserving function from $X$ to $Y$ is continuous.
The authors asked in Problem~4.1 whether every preserving function from a locally connected space that is either compact or sequential into a $T_2$-space is continuous.
These basic problems are still unsolved.

In this paper we are interested in the question what happens when a nontrivial product of functions is preserving.
We were surprised to see that just the preservation of connectedness for such a product function with a connected domain already yields its continuity.
For product functions that preserve compactness this is not so, as we will demonstrate by several counterexamples.

We then proceed by considering functions that preserve compactness or compactness and continua, as well as products of such functions.

%\section{Preliminaries}

\section{Connectedness-preserving product functions}

Assume that $f_i\from X_i\to Y_i$ is a function for every element $i$ of an index set $I$. Then $f = \prod_{i\in I} f_i$ denotes the product of the functions $\{f_i : i\in I\}$.
In other words, $f$ maps the cartesian product $X = \prod_{i\in I} X_i$ to the product $Y= \prod_{i\in I} Y_i$
so that for every $x = \langle x_i : i \in I \rangle \in X$ we have $f(x) = \langle f_i(x_i) : i \in I \rangle$.

Before formulating and proving the main result of this section, we present a simple technical lemma about products of connected spaces.

\begin{lemma}\label{lm:ctd}
Assume that $X_i$ for $i \in I$ are connected topological spaces and $X = \prod_{i\in I} X_i$.
As usual, $\pi_i\from X\to X_i$ denotes the $i^{th}$ projection map for every $i \in I$.
Assume, moreover, that we have $J \subset I$ with $|J| \ge 2$ and
for every index $i \in J$ a fixed non-empty subset $A_i \subset X_i$. Then the set
$$A = \bigcup_{i \in J}{\pi_i}^{-1}(A_i) \subset X$$ is connected.
\end{lemma}

\begin{proof}
Clearly, for any index $i \in I$ and for any point $a \in Y_i$ the set $\pi_i^{-1}(a) \subset X$ is connected.
Thus if we pick for each $i \in J$ a point $a_i \in A_i$ then the set $$B = \{x \in X : \exists \,i \in J\, \big(\pi_i(x) = a_i \big) \} = \bigcup_{i \in J}{\pi_i}^{-1}(a_i)$$
is also connected because $\bigcap_{i \in J}{\pi_i}^{-1}(a_i) \ne \emptyset$. Obviously, this connected set $B$
is included in $A$.

Now, for each $i \in J$ and $a \in A_i$ the connected set $\pi_i^{-1}(a) \subset A$ intersects $B \subset A$, as is witnessed by any point $x \in X$
that satisfies $\pi_i(x) = a$ and $\pi_j(x) = a_j$ for some $j \in J$ with $i \ne j$. (This is the point where we use $|J| \ge 2$.)
But $A$ is the union of all the connected sets  $\pi_i^{-1}(a)$ with $i \in J$ and $a \in A_i$, consequently $A$ is indeed connected.
\end{proof}

Now we are ready to present our main result.

\begin{theorem}\label{eerstestelling}
Let $X_i$ and $Y_i$ be a topological spaces for every $i\in I$,
where every $X_i$ is connected and every $Y_i$ is $T_1$.
Moreover,  $f_i\from X_i\to Y_i$ be a function for every $i \in I$ such that $f_i$ is non-constant for at least two different indices $i\in I$.
Then the product function $f=\prod_{i\in I} f_i \from \prod_{i\in I} X_i \to \prod_{i\in I} Y_i$ is connectedness-preserving if and only if
$f$, or equivalently, every $f_i$ is continuous.
\end{theorem}

\begin{proof}
Let $X = \prod_{i\in I} X_i$ and $Y= \prod_{i\in I} Y_i$. For every $i\in I$, let $\pi_i\from X\to X_i$ denote the projection as above,
and $\rho_i$ be the $i^{th}$ projection from $Y$ to $Y_i$.
Of course, we may assume without any loss of generality that $|I| \ge 2$ and $f_i$ is non-constant for every $i \in I$.
Obviously, we only have to prove that if $f$ is connectedness-preserving then it is continuous.

So, assume  that $f$ is connectedness-preserving and take a point $x = \langle x_i : i \in I \rangle \in X$ and a set $S \subset X$ such that $x \in \overline{S}$.
We have to show that then $f(x) \in \overline{f(S)}$. Assume, on the contrary, that $f(x) \notin \overline{f(S)}$, i.e. $f(x)$ has
an open neighborhood $U$ in $Y$ that misses $f(S)$. 

We may assume that $U = \bigcap_{i \in J} {\rho_i}^{-1}(U_i)$ where $J$ is a finite subset of $I$
with $|J| \ge 2$ and $U_i$ is an open proper subset of $Y_i$ for each $i \in J$. The latter assumption is justified because for every $i \in J$ there is
a point $z_i \in X_i$ with $f_i(z_i) \ne f_i(x_i)$ as $f_i$ is non-constant, and hence $Y_i \setminus \{f_i(z_i)\}$ is a neighborhood of $f_i(x_i)$
as $Y_i$ is $T_1$.

We may then apply our Lemma \ref{lm:ctd} to the finite index set $J$ and the non-empty sets $A_i = f_i^{-1}(Y_i \setminus U_i)$ to conclude that
$$A = \bigcup_{i \in J}{\pi_i}^{-1}(A_i) \subset X$$ is connected.

But then  $f(S) \cap U = f(S) \cap \bigcap_{i \in J} {\rho_i}^{-1}(U_i) = \emptyset$ clearly implies that $S \subset A$, consequently we have $x \in \overline{A}$,
hence the set $A \cup \{x\}$ is connected as well. On the other hand, we have $U \cap f(A \cup \{x\}) = \{f(x)\}$, hence $f(x)$ is an isolated point
of the non-singleton set $f(A \cup \{x\})$.
So, $f(A \cup \{x\})$ is disconnected, contradicting our assumption that $f$ is connectedness-preserving.
\end{proof}

It is an immediate consequence of this theorem that if a function $f$ with connected domain and $T_1$ range is connectedness-preserving but not continuous
then its square $f^2 = f \times f$ cannot be connectedness-preserving. The following simple example illustrates this.

\begin{example}\label{sin}
Let $f : \mathbb{I} \to [-1,1]$ be the function defined by $f(x) = \sin(1/x)$ for $x \ne 0$ and $f(0) = 0$.
Then $f$ is connectedness-preserving but $f^2$ is not.
\end{example}

It is very easy to check that $f$ is connectedness-preserving and it is trivially not continuous. But a direct verification of the fact
that $f^2$ is not connectedness-preserving requires some effort.

Of course, the condition in Theorem~\ref{eerstestelling} that there are at least two factors on which the respective functions are nonconstant is essential.
If all but one of them are constant then we are actually dealing with just one connectedness-preserving map and  from that nothing new can be concluded.

Also, any map whose domain is hereditarily disconnected is trivially connectedness-preserving, hence the assumption in Theorem~\ref{eerstestelling} that
all the factor spaces $X_i$ are connected is very natural.

\section{Compactness-preserving functions}

In this section we present several interesting examples of functions and product functions which, instead of connectedness preserving, are
compactness-preserving. Let us note first that Theorem~\ref{eerstestelling} badly fails if we replace in it both occurrences
of ``connected" by ``compact". This is because clearly any function $f$ with finite range is compactness-preserving, moreover so is any finite
power $f^n$ of it, hence any discontinuous map with finite range yields a counterexample.

We do have, however, a result that gives a simple condition under which compactness-preserving maps are automatically continuous.

\begin{theorem}\label{tm:fiber}
Assume that $f : X \to Y$ is a compactness-preserving function with compact fibers such that both $X$ and $\,Y$ are $T_2$ and $X$ is compact.
Then $f$ is continuous.
\end{theorem}

\begin{proof}
Clearly, we may assume that $f$ is a surjection, hence $Y$ is also compact.
Let $U$ be any open subset of $Y$ and assume that $f^{-1}(U)$ is not open in $X$. This means that there is a point $x \in f^{-1}(U)$
such that $x \in \overline{A}$ where $A = X \setminus f^{-1}(U)$.

Using that $X$ is compact $T_2$, and hence $T_4$, the family $\mathcal{C}$ of all closed, hence compact, neighborhoods of
the compact set $f^{-1}(f(x))$ in $X$ satisfies $\bigcap\,\mathcal{C} = f^{-1}(f(x))$. But every neighborhood of $f^{-1}(f(x))$ is also a neighborhood of $x$,
hence we have $C \cap A \ne \emptyset$ for all $C \in \mathcal{C}$.
It follows that $\{f(C) : C \in \mathcal{C}\} \cup \{Y \setminus U\}$ is a centered collection of closed sets in $Y$, hence there is a point $y$
in its intersection. But $f^{-1}(y) \subset X$ is also compact and clearly $f^{-1}(f(x)) \cap f^{-1}(y) = \emptyset$, hence there is some $C \in \mathcal{C}$
such that $C \cap  f^{-1}(y) = \emptyset$ as well. But this implies $y \notin f(C)$, contradicting our choice of $y$ and thus completing the proof.
\end{proof}

We now present an example which, in analogy with Example \ref{sin}, shows that the square of a compactness-preserving function may fail to be compactness-preserving.

\begin{example}\label{eersteex}
There is a compactness-preserving surjection $f\from \interval\to \omega{+}1$ such that its square $f^2\from \interval^2\to (\omega{+}1)^2$ is not compactness-preserving.
\end{example}

\begin{proof}
Let $t_0 = 0$ and for every $n \ge 1$, let $t_n = 1-\frac{1}{n+1}$. In addition, for every $n <\omega$ let $D_n \con [t_n,1)$ be a countable dense set such that if $n\not= m$ then $D_n\cap D_m=\emptyset$. Define $f\from \interval\to \omega{+}1$, as follows:
$$
f(s) = \begin{cases}
        n   &  (s\in D_n), \\
        0   &  (s\in [0,1)\setminus \bigcup_{n< \omega} D_n), \\
        \omega   &  (s=1).
       \end{cases}
$$
We first show that $f$ is compactness-preserving. To see this, let us note that any $C \subset \omega + 1$ is compact if $\omega \in C$. Consequently, if
$K$ is a compact subset of $\interval$ for which $1 \in K$ then $f(K)$ is compact. But if $1 \notin K$ then the compactness of $K$ implies that $K$ is
bounded below $1$, hence it may intersect only finitely many of the sets $D_n$. Thus, in this case $f(K)$ is finite.

We next will show that $f^2$ is not compactness-preserving. Pick a sequence $(d_n)_n$ in $D_0$ such that $d_n \to 1$. Moreover, for every $n$ let $e_n\in D_n$ be arbitrary. Observe that the sequence $(e_n)_n$ converges to 1 as well. Hence the sequence $(d_n,e_n)_n$ converges to (1,1) in $\interval^2$. Put $K = \{(d_n.e_n) : n <\omega\}\cup \{(1,1)\}$. Then $K$ is compact, but the infinite set
$$
    f(K) \subset \{(0,n) : n\in \omega\} \cup \{(\omega,\omega)\}
$$
is clearly not.
\end{proof}

We have noted above that any function $f\from \interval\to\interval$ which is both connectedness- and compactness-preserving is continuous.
So if we want to detect `bad behavior' in such functions which are connectedness-preserving, we have to relax the compactness-preserving condition.
It seems that the best we can hope for in this respect is to consider functions that are \emph{continuum-preserving}.
(That is, functions $f$ having the property that the image under $f$ of any continuum is a continuum.)

Clearly, any function $f\from \interval\to\interval$ is connectedness-preserving iff the image of any interval is an interval,
moreover it is continuum-preserving iff the image of any closed interval is a closed interval.
As any subinterval of  $\interval$ is an increasing union of closed intervals, it is easy to show that every  continuum-preserving
function $f\from \interval\to\interval$ is automatically connectedness-preserving.
Our next example yields a function $f\from \interval\to\interval$ that has the stronger of these two properties but its square misses even the weaker one.

\begin{example}\label{tweedeex}
There is a function $f\from \interval\to \interval$ that is continuum-preserving but for some arc $J\con \interval^2$ we have that $f^2(J)$ is not compact.
Consequently, $f^2$ is not connectedness-preserving.
\end{example}

\begin{proof}
Let $\{(r_n, s_n) : n < \omega\}$ enumerate all nontrivial open intervals in $(0,1)$ with rational endpoints.
For every $n$ let $K_n\con (r_n,s_n)$ be a Cantor set such that $K_n\cap K_m=\emptyset$ if $n \not= m$,
moreover $f_n\from K_n\to \interval$ be a surjection. We then define $f\from \interval \to \interval$, as follows:
$$
    f(s) = \begin{cases}
            f_n(s) & (s\in K_n, n< \omega), \\
            0      & (s\in \interval\setminus \bigcup_{n<\omega} K_n).
           \end{cases}
$$
Since every non-degenerate subinterval of $\interval$ contains one of the $K_n$'s it is mapped by $f$ onto $\interval$, hence $f$ is trivially continuum-preserving.

Let us fix a strictly increasing sequence $(t_n)_{n<\omega}$  in $(0,\frac{1}{2})$ converging to $\frac{1}{2}$.
We shall construct a polygonal arc $J$ in $[0,\frac{1}{2}]^2$ as the union of countably many vertical segments $V_n$ and countably many horizontal segments $H_n$ for $n < \omega$
and of a compactifying point as follows.
The segment $V_0$ has the form $\{0\}\times [0,\alpha_0]$, where $\alpha_0 < \frac{1}{4}$ and $f(\alpha_0) = t_0$. Since $f((0,\frac{1}{4}))=\interval$,
it is clear that such a value $\alpha_0$ exists. Observe that $f^2(V_0)= \{0\}\times \interval$.
The segment $H_0$ has the form $[0,\beta_0]\times \{\alpha_0\}$, where $\beta_0 < \frac{1}{4}$ and $f(\beta_0) = t_0$.
Observe that $f^2(H_0) = \interval\times \{t_0\}$. The segment $V_1$ has the form $\{\beta_0\}\times [\alpha_0,\alpha_1]$,
where $\alpha_1 - \alpha_0 < \frac{1}{8}$ and such that $f(\alpha_1) = t_1$. Observe that $f(V_1) = \{t_0\} \times \interval$.
The segment $H_1$ has the form $[\beta_0,\beta_1]\times \{\alpha_1\}$, where $\beta_1-\beta_0 < \frac{1}{8}$ and $f(\beta_1) = t_1$. Observe that $f^2(H_1) = \interval\times \{t_1\}$.

It should be clear how to continue this for all $n < \omega$ to obtain the strictly increasing sequences of values $\alpha_n, \, \beta_n$ with $f(\alpha_n) = f(\beta_n) = t_n$ satisfying both
$\alpha_{n^+1} - \alpha_n < 1/{2^{n+3}}$ and $\beta_{n^+1} - \beta_n < 1/{2^{n+3}}$, and then put $V_n = \{\beta_{n-1}\}\times [\alpha_{n-1},\alpha_n]$ and $H_n = [\beta_{n-1},\beta_n]\times \{\alpha_n\}$.
It is also clear that we shall then have $f^2(V_n) = \{t_{n-1}\} \times \interval$ and $f^2(H_n) = \interval\times \{t_n\}$ for all $n$.

Let $\beta = \sup \{\beta_n : n <\omega\}$ and $\alpha = \sup \{\alpha_n : n < \omega\}$, respectively.
Then clearly  $\alpha,\beta < \frac{1}{2}$, and the point $(\alpha,\beta)$ compactifies $\bigcup_{n<\omega} H_n \cup \bigcup_{n<\omega} V_n$, i.e.
$$
    J = \bigcup_{n<\omega} H_n \cup \bigcup_{n<\omega} V_n \cup \{(\beta,\alpha)\}
$$
is an arc.

Its image $f(J)$ contains the points $\{(1,t_n): n < \omega\}$ which converge to $(1,\frac{1}{2})$ and it also contains
the points$\{(t_n,1): n < \omega\}$ which converge to $(\frac{1}{2},1)$.
It is also clear from our construction that  $\{(1,\frac{1}{2}), (\frac{1}{2},1)\} \cap f(J\setminus \{(\beta,\alpha)\})=\emptyset$.
This, however, implies that $f^2(J)$ does not contain both $(1,\frac{1}{2})$ and $(\frac{1}{2},1)$ and therefore is not compact.

Consequently, $f^2$ is not continuum-preserving, and therefore it is not continuous. But then by Theorem \ref{eerstestelling}
it is not connectedness-preserving either.
\end{proof}

Next we show, using a slight variation of Example~\ref{tweedeex}, that a connectedness-preserving function $f\from\interval\to \interval$
is not necessarily continuum-preserving.

\begin{example}\label{derdeex}
There is a function $f\from \interval\to \interval$ which is connectedness-preserving but not continuum-preserving.
\end{example}

\begin{proof}
As in the proof of Example~\ref{tweedeex}, let $\{K_n : n < \omega\}$ be a pairwise disjoint sequence of Cantor sets in $(0,1)$ such that every nondegenerate interval in $\interval$ contains one of them.
Now, for every $n$ let $f_n$ be a surjection of $K_n$ not ont $\interval$ but onto the open interval $(0,1)$, and then define $f\from \interval \to \interval$ as follows:
$$
    f(s) = \begin{cases}
            f_n(s) & (s\in K_n, n< \omega), \\
            0      & (s=0), \\
            1      & (s=1),\\
            \frac{1}{2} & (s\in (0,1)\setminus \bigcup_{n<\omega} K_n).
           \end{cases}
$$
This function is connectedness-preserving for the same reason as the function from Example~\ref{tweedeex} is, but not continuum preserving since $f([0,\frac{1}{2}]) = [0,1)$.
\end{proof}

Our above observation that every  continuum-preserving
function $f\from \interval\to\interval$ is automatically connectedness-preserving obviously generalizes to any map from one connected LOTS to another.
As we may expect, this implication fails in general, even within the class of metric compacta. We present next two examples that demonstrate this.

\begin{example}\label{vierdeex}
For every non-degenerate indecomposable metrizable continuum $K$ there is a map $f\from K\to \interval$ which is continuum-preserving but not connectedness-preserving.
\end{example}

\begin{proof}
Let $K$ be any non-degenerate indecomposable metrizable continuum (for example, the pseudo-arc).
For every point $x\in K$, let $K_x$ denote the composant of $x$ in $K$. That is, $K_x$ is the union of all proper subcontinua of $K$ that contain $x$.
It is known that every $K_x$ is dense in $K$ and that every $K_x$ as well as the collection $\{K_x : x\in K\}$ has cardinality $\cont$.
For details, see Nadler~\cite{nadlerContinnuumTheory}.

Pick two distinct points $x$ and $y$ in $K$.
It is clear that there is a function $f\from K\to \interval$ which has the following properties:
\begin{enumerate}
\item $f$ is constant on each $K_z$ and, in particular, $f(K_x) = \{0\}$ and $f(K_y) = \{1\}$,
\item $f$ is surjective.
\end{enumerate}
Then $f$ maps every proper subcontinuum of $K$ to a single point, while $f(K) = \interval$.
On the other hand, it is clear from the density and connectivity of $K_x$ that $K_x\cup K_y$ is connected as well, while it is mapped onto $\{0,1\}$. Hence $f$ is as required.

\end{proof}

\begin{example}\label{plane}
Let $K$ be any non-degenerate metric continuum such that no countable subset of $K$ separates $K$ (for example, the square $\interval^2$).
Then there is a map $f\from K\to \interval$ which is continuum-preserving but not connectedness-preserving.
\end{example}

\begin{proof}
First split $K$ into two sets $A$ and $B$ which meet every Cantor subset of $K$ (a Bernstein decomposition). Then $A$ is dense and connected. Dense is clear.
If $A$ were disconnected, we could split it into two disjoint nonempty relatively open sets $U$ and $V$.
Then there are disjoint open subsets $E$ and $F$ of $K$ such that $E \cap A = U$ and $F \cap A = V$. But the closed set $S = K \setminus (E\cup F)$ separates $K$,
and therefore has to be uncountable.
Hence $S$ contains a Cantor set, and so meets $A$, a contradiction.

Now split $A$ into two nonempty disjoint sets $A_0$ and $A_1$.
Let $\mathcal{C}$ be the family of all Cantor sets in $K$. Then $C \cap B$ has size $\mathfrak{c}$ for every $C \in \mathcal{C}$.
By Kuratowski's disjoint refinement theorem, then
there is a family $\mathcal{S}$ consisting of $\mathfrak{c}$ many pairwise disjoint subsets of $B$, each of size $\mathfrak{c}$,
such that every member of $\mathcal{C}$ includes an element  of $\mathcal{S}$.

Now define the function $f: K \to \mathbb{I}$ by the following stipulations:
$f(A_0) = \{0\}, f(A_1) = \{1\}$, and
each member of $\mathcal{S}$ is mapped onto $\mathbb{I}$ by $f$.
Then $f$ is continuum-preserving since each non-degenerate subcontinuum of $K$ contains a Cantor set and hence a member of $\mathcal{S}$.
But $f$ is not connectedness-preserving because $f(A) = \{0,1\}$.
\end{proof}

Theorem~1.8 in \cite{GerlitsJuhaszSoukupSzentmiklossy04} says that if $X$ is a locally connected and Fr\'echet space then any preserving, i.e.
both compactness- and connectedness-preserving map on $X$ into a $T_2$ space is continuous.
Its proof is based on Lemma~1.7 in that paper which is easily seen to be true
for continuum-preserving functions in the special case in which the domain of the map is also locally compact. Hence the following is true:

\begin{theorem}\label{tm:Frechet}
Assume that $X$ is locally compact, locally connected and Fr\'echet, moreover $Y$ is $T_2$. Then every map
$f\from X\to Y$ that is both continuum-preserving and compactness-preserving is continuous.
\end{theorem}

Let us note that by the result of White~\cite{White71} mentioned in the introduction, the conclusion of Theorem  \ref{tm:Frechet} fails for any Tychonov space
$X$ that is not locally connected. Indeed, this is because any preserving map is trivially continuum-preserving.

Our last example shows that, at least under the Continuum Hypothesis (abbreviated: \CH), we can even get a locally connected and compact $T_2$ space for which
the conclusion of Theorem  \ref{tm:Frechet} fails. Of course, such a space cannot have the Fr\'echet property, and our example, not surprisingly, will be
kind of ``anti-Fr\'echet".

\begin{example}[\CH]\label{zesdeex}
There are a locally connected Hausdorff continuum $X$ and a function $f\from X\to \interval$ such that $f$ is both compactness- and continuum-preserving but not connectedness-preserving.
\end{example}

\begin{proof}
Let $X$ be the locally connected $T_2$ continuum of weight $\cont$ containing no convergent sequences that was constructed, using CH, 
by van Mill~\cite{MillLocallyConnected2002}. 
(Under the stronger assumption $\Diamond$, even a 1-dimensional, locally connected, and hereditarily separable continuum without convergent sequences was constructed by Hart and Kunen~\cite{HartKunen09}.)

It is well-known that then
every infinite compact subspace of $X$ has cardinality $2^{\omega_1}=2^\cont$ and there are at most $2^\cont$ such compacta since the weight of $X$ is $\cont$. We can consequently partition 
$X$ into a collection $\{A_t : t\in\interval\}$
of sets $A_t$ intersecting every infinite closed subset of $X$. (In fact, such a family can even be chosen to be of cardinality $2^\cont$ by Kuratowski's disjoint refinement theorem.) 
Let then $f\from X\to \interval$ be the function that sends every $A_t$ onto the point $t$. Clearly, $f$ sends every infinite compact subset of $X$ onto $\interval$, and hence is both compactness- and continuum-preserving.
Since every non-empty open set in $X$ includes an infinite compact subset, it is clear that every $A_t$ is dense in $X$.
 
We claim that from the construction in van Mill~\cite{MillLocallyConnected2002} it follows that our space $X$ cannot be separated by any finite subset (the same is true for the construction in \cite{HartKunen09}). But if this is true, then every dense set $A_t$ intersects every closed set that separates $X$, 
and this means that $A_t$ is connected. This implies then that $A_0\cup A_1$ is connected, while its image $f(A_0\cup A_1) = \{0,1\}$ is not.

To verify our claim, let us assume that the space $X$ in \cite{MillLocallyConnected2002} has a cut point, say $x$.  (The same proof will show that no finite set separates $X$.) We adopt the notation in \cite{MillLocallyConnected2002}. Hence $X= M = M_{\omega_1}$ is the limit of the inverse system of length $\omega_1$
presented in~\cite{MillLocallyConnected2002}, in which $M_\alpha$ is the Hilbert cube for each $\alpha < \omega_1$.
Since all bonding maps of this inverse system
are monotone, so is the projection $\pi=\pi^{\omega_1}_0\from M\to M_0$. 

Since $x$ is a cut point of $M$, we may write $M\setminus \{x\}$ as $U\cup V$, where $U$ and $V$ are nonempty pairwise disjoint open sets. Then for every 
$q\in M_0\setminus \{\pi(x)\}$ we have that $\pi^{-1}(q)$ is a continuum contained either in $U$ or in $V$.
So, if we put $E= \{q\in M_0\setminus \{p\}: \pi^{-1}(q)\con U\}$ and $F= \{q\in M_0\setminus \{p\}: \pi^{-1}(q)\con V\}$, respectively, then $E$ and $F$ 
would demonstrate that $\pi(x)$ is a cut point of the Hilbert cube $M_0$, which is absurd.
\end{proof}

%\bibliographystyle{\br{michael}}
%\bibliographystyle{alpha}
%\bibliography{\br{strings},\br{mill},\br{publ},\br{eric},mybib}

\def\cprime{$'$}
\makeatletter \renewcommand{\@biblabel}[1]{\hfill[#1]}\makeatother

\end{document}